\documentclass[reqno]{amsart}
\usepackage{hyperref}
\usepackage{graphicx}
\usepackage{amssymb}
\usepackage{amsfonts}
\usepackage{amsmath}
\usepackage[a4paper,top=3.5cm,bottom=3cm,left=2cm,right=2cm]{geometry}

\begin{document}
\title[\hfilneg positive solutions of conformable fractional differential equations  \hfil ]
{Existence of positive solutions for a class of conformable fractional differential equations with integral boundary conditions and a parameter}

\author[F. Haddouchi]
{Faouzi Haddouchi}

\address{Faouzi Haddouchi \newline
Department of Physics, University of Sciences and Technology of
Oran-MB \newline
El Mnaouar, BP 1505, 31000 Oran, Algeria
\newline
Laboratory of Fundamental and Applied Mathematics of Oran,\newline
Department of Mathematics, University of Oran 1 Ahmed Benbella,\newline
31000 Oran,
Algeria}
\email{fhaddouchi@gmail.com}


\begin{abstract}
In this paper, we study the existence of positive solutions for a class of conformable fractional differential equations with integral boundary conditions. By using the properties of the Green's function and the fixed point theorem in a cone, we obtain some existence results of positive solution. we also provide some examples to illustrate our results.
\end{abstract}

\maketitle
\numberwithin{equation}{section}
\newtheorem{theorem}{Theorem}[section]
\newtheorem{corollary}[theorem]{Corollary}
\newtheorem{lemma}[theorem]{Lemma}
\newtheorem{remark}[theorem]{Remark}
\newtheorem{definition}[theorem]{Definition}
\newtheorem{example}[theorem]{Example}

\allowdisplaybreaks
\section{Introduction}
Fractional calculus and fractional differential equations are experiencing a rapid development. There are several concepts of fractional derivatives, some classical, such as Riemann-Liouville or Caputo definitions, and some novel, such as
conformable fractional derivative \cite{Khalil}, $\beta$-derivative \cite{Atangana}, or a new definition \cite{Caputo,Losada}.
Recently, the new conformable fractional derivative definition given by  \cite{Khalil,Abdeljawad,Horani} has drawn much interest from many researchers \cite{Yang,Zhao,Zhou,Anderson1,Anderson2,Weberszpil,Katugampola}.
Recent results on conformable fractional differential equations can also be seen in \cite{Bayour,Al-Rifae,Asawasamrit}

In 2017, X. Dong et al.\cite{Dong} studied the existence and multiplicity of positive solutions for the following conformable fractional differential equation with $p$-Laplacian operator
\begin{equation*}
D^{\alpha}(\phi_{p}(D^{\alpha}u(t)))=f(t,u(t)), \ 0<t<1,
\end{equation*}
\begin{equation*}
u(0)=u(1)=D^{\alpha}u(0)= D^{\alpha}u(1)=0,
\end{equation*}
where $1<\alpha\leq 2$ is a real number, $D^{\alpha}$ is the conformable fractional derivative, $\phi_{p}(s)=|s|^{p-2}s$, $p>1$, $\phi_{p}^{-1}=\phi_{q}$, ${1/p}+{1/q}=1$, and $f:[0,1]\times[0,+\infty)\rightarrow [0,+\infty$ is continuous. By the use of an approximation method and fixed point theorems on cone, some existence results are established.

In \cite{Batarfi}, the authors considered the following three-point boundary value problem for conformable fractional differential equation

\begin{equation*}
D^{\alpha}(D+\lambda)x(t)=f(t,x(t)), \ t \in [0,1],
\end{equation*}
\begin{equation*}
x(0)=0, \  x^{\prime}(0) = 0, \ x(1) =\beta x(\eta),
\end{equation*}
where $D^{\alpha}$ is the conformable fractional derivative of order $\alpha\in(1,2]$, $D$ is the ordinary derivative, $f:[0,1]\times \mathbb{R}\rightarrow \mathbb{R}$ is a known continuous function, $\lambda$ and $\beta$ are real numbers, $\lambda>0$, and $\eta\in(0,1)$.
The existence results  are obtained by means of Krasnoselskii's fixed
point theorem and the classical Banach fixed point theorem.

In \cite{Anderson3},  D. R. Anderson et al., considered the following conformable fractional-order boundary value problem with Sturm-Liouville boundary conditions

\begin{equation*}
-D^{\beta}D^{\alpha}x(t)=f(t,x(t)), \ 0\leq t\leq1,
\end{equation*}
\begin{equation*}
\gamma x(0)-\delta D^{\alpha}x(0)=0=\eta x(1)+\zeta D^{\alpha}x(1),
\end{equation*}
where $\alpha, \beta \in (0,1]$ and the derivatives are conformable fractional derivatives, with $\gamma,\delta, \eta, \zeta \geq0$ and $d=\eta\delta+\gamma\zeta+\gamma\eta/\alpha>0$.
By employing a functional compression expansion fixed point theorem due to Avery, Henderson, and O’Regan, they proved the existence of positive solution.

In a recent paper \cite{He}, by using the well-known topological transversality
theorem, L. He et al., obtained the existence of solutions for fractional differential
equation

\begin{equation*}
D^{\alpha}x(t)=f(t,x(t),D^{\alpha-1}x(t)), \ t \in [0,1],
\end{equation*}
with one of the following boundary value conditions
\begin{equation*}
x(0)=A, \ D^{\alpha-1}x(1)=B; \ \text{or} \ D^{\alpha-1}x(0)=A, \ x(1)=B,
\end{equation*}
where $\alpha\in(1,2]$ is a real number, $D^{\alpha}x(t)$ is the conformable fractional order derivative of a function $x(t)$, and $f:[0,1]\times \mathbb{R}^{2}\rightarrow \mathbb{R}$ is a continuous function. The existence results of solutions to the problem are obtained under $f$ which satisfies some barrier strip conditions.

In the same year, Q. Song et al. \cite{Song} investigated the following fractional Dirichlet boundary value problem

\begin{equation*}
D^{\alpha}x(t)=f(t,x(t),D^{\alpha-1}x(t)), \ t \in [0,1],
\end{equation*}
\begin{equation*}
x(0)=A, \ x(1)=B,
\end{equation*}
where $1<\alpha\leq 2$, $D^{\alpha}x(t)$ is the conformable fractional derivative, and $f:[0,1]\times \mathbb{R}^{2}\rightarrow \mathbb{R}$ is a continuous function. The existence results of solutions to the problem are obtained under $f$ satisfying some sign conditions.

Very recently, in 2018, W. Zhong and L. Wang \cite{Zhong} discussed the existence of positive solutions of the conformable fractional differential equation
\begin{equation*}
D^{\alpha}x(t)+f(t,x(t))=0, \ t \in [0,1],
\end{equation*}
subject to the boundary conditions
\begin{equation*}
x(0)=0, \ x(1)=\lambda \int_{0}^{1}x(t)dt,
\end{equation*}
where the order $\alpha$ belongs to $(1,2]$, $D^{\alpha}x(t)$ denotes the conformable fractional derivative of a function $x(t)$ of order $\alpha$, and $f:[0,1]\times [0,\infty)\rightarrow [0,\infty)$ is a continuous function. By employing a fixed point theorem in a cone, they established
some criteria for the existence of at least one positive solution.

Inspired and motivated by the above recent works, we intend in the present paper
to study the existence of positive solutions to boundary value problem of conformable fractional differential equation

\begin{equation}\label{eq1}
D^{\alpha}x(t) + f(t, x(t)) = 0, \ t \in [0,1],
\end{equation}
\begin{equation}\label{eq2}
x(0) =0, \ x(1) =\lambda \int_{0}^{\eta}x(t)dt,
\end{equation}
where $D^{\alpha}x(t)$ denotes the conformable fractional derivative of $x$ at $t$ of order $\alpha$, $\alpha\in(1,2]$, $\eta\in(0, 1]$, $f\in C([0, 1]\times[0,\infty),[0,\infty))$ is a continuous function, and the parameter $\lambda$ is a positive constant.

For the case of $\eta=1$, problem \eqref{eq1} and \eqref{eq2} reduces to the problem studied by Zhong and Wang in \cite{Zhong}. Our approach is similar to that used in \cite{Zhong}, i.e., fixed point theorem in a cone, lower and upper bounds for the Green’s function are employed as the main tool of analysis. It should noticed that our results seem more natural than those in \cite{Zhong}, and in this case,
the results in \cite{Zhong} are special cases of those in this paper. Our work extends and complements the results in \cite{Zhong}. It is worth pointing out that the obtained Green's function in this work is singular at $s=0$.

The rest of this paper is arranged as follows:\\
In Section 2, we present the necessary definitions
and we give some lemmas in order to prove our main results. In particular, we state some properties of the Green's function associated with BVP \eqref{eq1} and \eqref{eq2}. In Section 3, some sufficient conditions are established for the existence of positive solution to our BVP when $f$ is superlinear or sublinear. Finally, two examples are also included to illustrate the main results.\\
\section{Preliminaries and lemmas}

In this section, we preliminarily give some definitions and results concerning conformable
fractional derivative. These results can be found in the recent literature, see \cite{Khalil,Abdeljawad,Dong}.

\begin{definition} \label{def 2.1}\emph{(\cite{Khalil,Abdeljawad})}
Let $\alpha \in (n,n+1]$ and $f$ be a $n$-differentiable function at $t>0$, then the
fractional conformable derivative of order $\alpha$ at $t>0$ is given by

\begin{equation*}
D^{\alpha}f(t)=\lim_{\epsilon\rightarrow 0}\frac{f^{(n)}(t+\epsilon t^{n+1-\alpha})-f^{(n)}(t)}{\epsilon},
\end{equation*}
provided the limits of the right side exists.
\end{definition}
If $f$ is $\alpha$-order differentiable on $(0, a)$, $a > 0$, and $\lim_{t\rightarrow 0^{+}}D^{\alpha}f(t)$ exists, then define
\[ D^{\alpha}f(0)=\lim_{t\rightarrow 0^{+}}D^{\alpha}f(t).\]

\begin{definition}  \label{def 2.2}\emph{(\cite{Khalil,Abdeljawad})}
Let $\alpha \in (n,n+1]$ and set $\beta=\alpha-n$. Then, the fractional derivative of a function
$f : [0,\infty)\rightarrow \mathbb{R}$ of order $\alpha$, where $f^{(n)}(t)$ exists, is defined by
\[D^{\alpha}f(t)=D^{\beta}f^{(n)}(t).\]
\end{definition}

\begin{lemma}  \label{lem 2.1}\emph{(\cite{Khalil,Dong})}
Let $\alpha \in (n,n+1]$ and $t>0$. The function $f(t)$ is $(n+1)$-differentiable if and only if
$f$ is $\alpha$-differentiable, moreover, $D^{\alpha}f(t)=t^{n+1-\alpha}f^{(n+1)}(t)$.
\end{lemma}

\begin{definition}  \label{def 2.3}\emph{(\cite{Abdeljawad})}
Let $\alpha$ be in $(n, n+1]$. The fractional integral of a function $f: [0,\infty)\rightarrow \mathbb{R}$ of order $\alpha$ is defined by
\begin{equation*}
I^{\alpha}f(t)=\frac{1}{n!}\int_{0}^{t}(t-s)^{n}s^{\alpha-n-1}f(s)ds.
\end{equation*}
\end{definition}

\begin{lemma} \label{lem 2.2}\emph{(\cite{Khalil,Abdeljawad,Dong})}
Let $\alpha$ be in $(n, n + 1]$. If $f$ is a continuous function on $[0,\infty)$, then,
for all $t > 0$, $D^{\alpha}I^{\alpha}f(t)=f(t)$.
\end{lemma}

\begin{lemma} \label{lem 2.3}\emph{(\cite{Dong})}
Let $\alpha \in (n, n + 1]$, $f$ be a $\alpha$-differentiable function at $t > 0$, then $D^{\alpha}f(t)=0$ for $t\in (0,\infty)$ if and only if $f (t) = a_{0} + a_{1}t +...+ a_{n-1}t^{n-1}+ a_{n}t^{n}$, where $a_{k}\in \mathbb{R}$, for $k =0,1,...,n$.
\end{lemma}

\begin{lemma} \label{lem 2.4}\emph{(\cite{Abdeljawad,Dong})}
Let $\alpha$ be in $(n, n + 1]$. If $D^{\alpha}f(t)$ is continuous on $[0,\infty)$, then
$I^{\alpha}D^{\alpha}f(t) = f (t) + c_{0} + c_{1}t +...+ c_{n}t^{n}$
for some real numbers $c_{k}$ , $k = 0,1,...,n$.
\end{lemma}

In order to study boundary value problem \eqref{eq1}-\eqref{eq2}, we consider first the linear equation

\begin{equation}\label{eq3}
D^{\alpha}x(t)+h(t) = 0, \ t \in [0,1],
\end{equation}
where $\alpha\in(1,2]$ and $h\in C([0,1])$.

\begin{lemma}\label{lem 2.5}
If $\lambda{\eta^{2}}\neq2$, then the unique solution of \eqref{eq3} subject to the boundary
conditions \eqref{eq2} is given by
$$ x(t) = \int_{0}^{1}\mathcal{K}(t,s) h(s) ds,$$
where
\begin{equation}\label{eq4}
\mathcal{K}(t,s)=\mathcal{G}(t,s)+\frac{\lambda{t}}{2-\lambda{\eta^{2}}}\mathcal{H}(\eta,s),
\end{equation}

\begin{equation}\label{eq5}
\mathcal{G}(t,s)=\begin{cases}(1-t)s^{\alpha-1}, & 0 \leq s \leq t \leq 1; \\
t(1-s)s^{\alpha-2}, & 0 \leq t \leq s \leq 1,
\end{cases}
\end{equation}
and
\begin{equation}\label{eq6}
\mathcal{H}(t,s)=\begin{cases}(2t-t^{2}-s)s^{\alpha-1}, & 0 \leq s \leq t \leq 1; \\
t^{2}(1-s)s^{\alpha-2}, & 0\leq t \leq s \leq 1.
\end{cases}
\end{equation}
\end{lemma}

\begin{proof}
From Lemma \ref{lem 2.4}, we may reduce \eqref{eq3} to an equivalent integral equation,
\begin{gather*}
\begin{aligned}
x(t)& =-I_{\alpha}h(t)+c_{0}+c_{1}t \\
&  =-\int_{0}^{t}(t-s)s^{\alpha-2}h(s)ds+c_{0}+c_{1}t,
\end{aligned}
\end{gather*}
for some $c_{0}, c_{1}\in \mathbb{R}$.
By \eqref{eq2}, we get $c_{0}=0$ and $c_{1}=I_{\alpha}h(1)+x(1)$.
Hence
\begin{eqnarray*}
x(t)& =&-I_{\alpha}h(t)+tI_{\alpha}h(1)+tx(1) \\
& =&-\int_{0}^{t}(t-s)s^{\alpha-2}h(s)ds+t\int_{0}^{1}(1-s)s^{\alpha-2}h(s)ds+tx(1)\\
&=&-\int_{0}^{t}(t-s)s^{\alpha-2}h(s)ds+t\int_{0}^{t}(1-s)s^{\alpha-2}h(s)ds\\
&&+t\int_{t}^{1}(1-s)s^{\alpha-2}h(s)ds+tx(1)\\
&=&\int_{0}^{t}(1-t)s^{\alpha-1}h(s)ds+\int_{t}^{1}t(1-s)s^{\alpha-2}h(s)ds+tx(1).
\end{eqnarray*}
So
\begin{equation} \label{eq7}
x(t)= \int_{0} ^{1} \mathcal{G}(t,s)h(s) ds + tx(1).
\end{equation}

Moreover, in checking the second boundary condition, we get
\begin{eqnarray*}
x(1)& =&{\lambda}\int_{0}^{\eta}x(t)dt\\
& =&{\lambda}\int_{0}^{\eta}\big[-I_{\alpha}h(t)+tI_{\alpha}h(1)+tx(1)\big]dt \\
& =&-{\lambda}\int_{0}^{\eta}\Bigg(\int_{0}^{t}(t-s)s^{\alpha-2}h(s)ds\Bigg)dt+\frac{\lambda \eta^{2}}{2}I_{\alpha}h(1)+\frac{\lambda \eta^{2}}{2}x(1)\\
& =&-\frac{\lambda}{2}\int_{0}^{\eta}(\eta-s)^{2}s^{\alpha-2}h(s)ds+\frac{\lambda \eta^{2}}{2}I_{\alpha}h(1)+\frac{\lambda \eta^{2}}{2}x(1),
\end{eqnarray*}
which implies
\begin{equation*}
 x(1) =-\frac{\lambda}{2-\lambda{\eta^{2}}} \int_{0}^{\eta} (\eta-s)^{2}s^{\alpha-2}h(s)ds+\frac{\lambda{\eta^{2}}}{2-\lambda{\eta^{2}}}I_{\alpha}h(1).
\end{equation*}
Substituting the value of $x(1)$ in \eqref{eq7}, we get
\begin{eqnarray*}
x(t)& =&\int_{0} ^{1} \mathcal{G}(t,s)h(s) ds-\frac{\lambda{t}}{2-\lambda{\eta^{2}}} \int_{0}^{\eta} (\eta-s)^{2}s^{\alpha-2}h(s)ds\\
&&+\frac{\lambda{\eta^{2}}t}{2-\lambda{\eta^{2}}}\int_{0}^{1}(1-s)s^{\alpha-2}h(s)ds\\
& =& \int_{0} ^{1}  \mathcal{G}(t,s)h(s) ds-\frac{\lambda{t}}{2-\lambda{\eta^{2}}} \int_{0}^{\eta} (\eta-s)^{2}s^{\alpha-2}h(s)ds\\
&&+\frac{\lambda{\eta^{2}}t}{2-\lambda{\eta^{2}}}\int_{0}^{\eta}(1-s)s^{\alpha-2}h(s)ds+
\frac{\lambda{\eta^{2}}t}{2-\lambda{\eta^{2}}}\int_{\eta}^{1}(1-s)s^{\alpha-2}h(s)ds\\
& =&\int_{0} ^{1}  \mathcal{G}(t,s)h(s) ds+\frac{\lambda{t}}{2-\lambda{\eta^{2}}} \int_{0}^{\eta}s^{\alpha-2}\Big[\eta^{2}(1-s)-(\eta-s)^{2}\Big]h(s)ds\\
&&+\frac{\lambda{t}}{2-\lambda{\eta^{2}}}\int_{\eta}^{1}\eta^{2}(1-s)s^{\alpha-2}h(s)ds\\
& =&\int_{0} ^{1}  \mathcal{G}(t,s)h(s) ds+\frac{\lambda{t}}{2-\lambda{\eta^{2}}} \int_{0}^{\eta}s^{\alpha-1}(2\eta-{\eta}^{2}-s)h(s)ds\\
&&+\frac{\lambda{t}}{2-\lambda{\eta^{2}}}\int_{\eta}^{1}\eta^{2}(1-s)s^{\alpha-2}h(s)ds\\
& =&\int_{0} ^{1}  \mathcal{G}(t,s)h(s) ds+\frac{\lambda{t}}{2-\lambda{\eta^{2}}}\int_{0}^{1} \mathcal{H}(\eta,s)h(s)ds.
\end{eqnarray*}
The proof is therefore complete.
\end{proof}
We point out here that \eqref{eq5}-\eqref{eq6} become the usual Green’s function when $\alpha=2$.
\begin{lemma}\label{lem 2.6}
Let $\theta\in (0,\frac{1}{2})$ be fixed.
For $ \mathcal{G}(t,s)$ and $ \mathcal{H}(t,s)$ given in \eqref{eq5}-\eqref{eq6}, we have the following bounds.

\begin{itemize}
\item[(i)] $\theta^{2} \mathcal{G}(s,s) \leq  \mathcal{G}(t,s) \leq  \mathcal{G}(s,s),$  \ for all \  $(t,s) \in [\theta , 1- \theta] \times (0,1];$
\item[(ii)] $\rho(t) \mathcal{G}(s,s) \leq  \mathcal{H}(t,s) \leq  \mathcal{G}(s,s),$  \ for all \  $(t,s) \in (0,1] \times (0,1],$
where $ \mathcal{G}(s,s)=(1-s)s^{\alpha-1},$ and
\begin{equation*}
\rho(t)=\min\big\{t^{2}, t(1-t)\big\}=\begin{cases} t^{2},& t\leq \frac{1}{2};  \\
t(1-t), & t \geq \frac{1}{2}.
\end{cases}
\end{equation*}
\item[(iii)] $\theta^{2}\mathcal{G}(s,s) \leq \mathcal{H}(t,s) \leq \mathcal{G}(s,s),$  \ for all \  $(t,s) \in [\theta , 1- \theta] \times (0,1],$
  \end{itemize}
\end{lemma}
\begin{proof}
{\rm (i)} From Lemma (2.5) in \cite{Zhong}, we have
\[t(1-t) \mathcal{G}(s,s) \leq  \mathcal{G}(t,s) \leq  \mathcal{G}(s,s),  \ \forall \ (t,s)\in (0,1]\times (0,1].\]
Therefore if $\theta\in (0,\frac{1}{2})$, then $ \mathcal{G}(t,s)$ satisfies
\[\theta^{2} \mathcal{G}(s,s) \leq  \mathcal{G}(t,s) \leq  \mathcal{G}(s,s),  \ \  \forall (t,s) \in [\theta , 1- \theta] \times (0,1].\]

{\rm (ii)} If $ s \leq t $, then from \eqref{eq6} we have
\begin{equation}\label{eq8}
\begin{split}
  \mathcal{H}(t,s) &= (2t-t^{2}-s)s^{\alpha-1}\\
  & =\big[-(t^{2}-2t)-s\big]s^{\alpha-1}\\
   & =\big(-[(t-1)^{2}-1]-s\big)s^{\alpha-1}\\
    & =\big[(1-s)-(1-t)^{2}\big]s^{\alpha-1}\\
   & \leq (1-s)s^{\alpha-1}. \\
 \end{split}
 \end{equation}\\
 On the other hand, we have
 \begin{equation}\label{eq9}
 \begin{split}
  \mathcal{H}(t,s) &= (2t-t^{2}-s)s^{\alpha-1}\\
  \mathcal{H}(t,s) &= \big[t(1-t)+(t-s)\big]s^{\alpha-1}\\
   & \geq t(1-t)s^{\alpha-1}\\
    & \geq (1-s)s^{\alpha-1}t(1-t).\\
  \end{split}
  \end{equation}\\
If $ t \leq s $, from \eqref{eq6}, we have
   \begin{equation}\label{eq10}
   \begin{split}
   \mathcal{H}(t,s) &=  t^{2}(1-s)s^{\alpha-2}  \\
     & \leq  t(1-s)s^{\alpha-2} \\
     & =\frac{t}{s}(1-s)s^{\alpha-1} \\
     & \leq (1-s)s^{\alpha-1},\\
    \end{split}
    \end{equation}
    and,
   \begin{equation}\label{eq11}
      \begin{split}
       \mathcal{H}(t,s) &=  t^{2}(1-s)s^{\alpha-2}\\
        & \geq  t^{2}(1-s)s^{\alpha-2}s \\
        & =t^{2}(1-s)s^{\alpha-1}.
       \end{split}
       \end{equation}\\
From \eqref{eq8}, \eqref{eq9}, \eqref{eq10} and \eqref{eq11}, we have
\[ \rho (t)(1-s)s^{\alpha-1} \leq \mathcal{H}(t,s) \leq (1-s)s^{\alpha-1},\ for\  all \ (t,s) \in (0,1] \times (0,1].\]
{\rm (iii)} It follows immediately from {\rm (ii)}.
\end{proof}

\begin{lemma}\label{lem 2.7}
Let $ \theta \in (0, \frac{1}{2})$ be fixed and $0\leq\lambda<2/{\eta^{2}}$. If $ h(t) \in C ([0,1],[0,\infty ))$, then the unique solution of \eqref{eq3} subject to the boundary conditions \eqref{eq2} is nonnegative and satisfies
\[ \min_{t\in [\theta, 1-\theta]} x(t) \geq \theta^{2}\|x\|.\]
\end{lemma}

\begin{proof}
From Lemma \ref{lem 2.5} and Lemma \ref{lem 2.6}, $ x(t) $ is nonnegative for $ t \in [0,1] $, and we get
\begin{equation*}
\begin{split}
x(t) &= \int_{0}^{1} \mathcal{K}(t,s) h(s)ds \\
 &=\int_{0} ^{1} \mathcal{G}(t,s)h(s) ds+\frac{\lambda{t}}{2-\lambda{\eta^{2}}}\int_{0}^{1}\mathcal{H}(\eta,s)h(s)ds\\
&\leq \int_{0} ^{1} \mathcal{G}(s,s)h(s) ds+\frac{\lambda}{2-\lambda{\eta^{2}}}\int_{0}^{1}\mathcal{H}(\eta,s)h(s)ds.
\end{split}
\end{equation*}
Then
\begin{equation}\label{eq12}
\|x\| \leq \int_{0} ^{1} \mathcal{G}(s,s)h(s) ds+\frac{\lambda}{2-\lambda{\eta^{2}}}\int_{0}^{1}\mathcal{H}(\eta,s)h(s)ds .
\end{equation}

On the other hand, from Lemma \ref{lem 2.6} for any $ t\in [\theta, 1-\theta] $, we have
\begin{equation} \label{eq13}
\begin{split}
x(t)&= \int_{0} ^{1} \mathcal{G}(t,s)h(s) ds+\frac{\lambda{t}}{2-\lambda{\eta^{2}}}\int_{0}^{1}\mathcal{H}(\eta,s)h(s)ds\\
&\geq \theta^{2}\int_{0} ^{1} \mathcal{G}(s,s)h(s) ds+\frac{\lambda{t^{2}}}{2-\lambda{\eta^{2}}}\int_{0}^{1}\mathcal{H}(\eta,s)h(s)ds \\
&\geq \theta^{2}\int_{0} ^{1} \mathcal{G}(s,s)h(s) ds+\frac{\lambda{{\theta}^{2}}}{2-\lambda{\eta^{2}}}\int_{0}^{1}\mathcal{H}(\eta,s)h(s)ds \\
&=\theta^{2}\left[ \int_{0} ^{1} \mathcal{G}(s,s)h(s) ds+\frac{\lambda}{2-\lambda{\eta^{2}}}\int_{0}^{1}\mathcal{H}(\eta,s)h(s)ds\right]\\
&\geq \theta^{2}\|x\|.
\end{split}
\end{equation}
From \eqref{eq12} and \eqref{eq13}, we obtain $$\min_{t\in [\theta, 1-\theta]} x(t) \geq  \theta^{2}\|x\|.$$
\end{proof}

In order to prove our main results, the following well known fixed point theorems are needed in the forthcoming analysis \cite{Deimling,Lan,Amann}.

\begin{lemma}\label{lem 2.8}
Let $\mathcal{B}$ be a Banach space, and let $\mathcal{P}\subseteq \mathcal{B}$, be a cone, and $\Omega_{1}$, $\Omega_{2}$ two bounded open balls of $\mathcal{B}$ centered at the origin with $\overline{\Omega}_{1}\subset{\Omega_{2}}$. Assume that
$A:\mathcal{P}\cap (\overline{\Omega}_{2} \backslash \Omega_{1})\rightarrow \mathcal{P}$ is a completely
continuous operator such that
\begin{itemize}
\item[(C1)]
$\left\Vert Ax\right\Vert \leq \left\Vert x\right\Vert ,$ $x\in\mathcal{P}\cap
\partial \Omega _{1}$.
\item[(C2)]
There exists $ \varphi\in \mathcal{P}\backslash \{0\}$ such that $x\neq Ax+\lambda{\varphi}$ for $x\in\mathcal{P}\cap \partial \Omega _{2}$ and $\lambda>0$.
\end{itemize}
Then $A$ has a fixed point in $\mathcal{P}\cap (\overline{\Omega}_{2} \backslash \Omega_{1})$. The same conclusion remains valid if {\rm (C1)} holds on $\mathcal{P}\cap
\partial \Omega _{2}$ and {\rm (C2)} holds on  $\mathcal{P}\cap \partial \Omega _{1}$.
\end{lemma}
\section{Existence results}
Throughout this section, we assume that
\begin{itemize}
\item[(H)] $f\in C([0, 1]\times[0,\infty),[0,\infty))$, and the parameter $\lambda \in [0,\frac{2}{\eta^{2}}).$
\end{itemize}

Let $E=C([0,1],\ \mathbb{R})$ be the Banach space endowed with the sup norm \[\|x\|=\\sup_{t\in[0, 1]}|x(t)|.\]
Let $ \theta \in (0, \frac{1}{2})$, define the cone $\mathcal{P}$ in $E$ by

$$\mathcal{P}= \left\lbrace x \in E,\ x \geq 0 : \min_{t\in [\theta, 1-\theta]} x(t) \geq  \theta^{2}\|x\| \right\rbrace.$$
Given a positive number $r$, define the subset $\partial\Omega_{r}$ of $E$ by
\[ \partial\Omega_{r}=\big\{x\in E: \|x\|<r\big \},\]
and also, define the operator $\mathcal{A}:E\rightarrow E$ by
\begin{equation}\label{eq14}
(\mathcal{A}x)(t)= \int_{0}^{1}\mathcal{K}(t,s)f(s,x(s))ds.
\end{equation}

\begin{lemma} \label{lem 3.1}
If the hypothesis \rm{(H)} holds, then $\mathcal{A}(\mathcal{P})\subset \mathcal{P}.$
\end{lemma}
\begin{proof}
By \eqref{eq14} and Lemma \ref{lem 2.7}, we have $\mathcal{A}(\mathcal{P})\subset \mathcal{P}.$
\end{proof}
In order to discuss the complete continuity of the operator $\mathcal{A}$, denote the operator $\mathcal{A}$ by
\[\mathcal{A}=\mathcal{A}_{1}+\mathcal{A}_{2},\]
where the operators $\mathcal{A}_{1}$ and $\mathcal{A}_{2}$ are defined, respectively by
\begin{equation}\label{eq15}
(\mathcal{A}_{1}x)(t)= \int_{0}^{1}\mathcal{G}(t,s)f(s,x(s))ds,
\end{equation}
and
\begin{equation}\label{eq16}
(\mathcal{A}_{2}x)(t)= \frac{\lambda{t}}{2-\lambda{\eta^{2}}}\int_{0}^{1} \mathcal{H}(\eta,s)f(s,x(s))ds.
\end{equation}
By Lemma \ref{lem 2.7}, it follows that $\mathcal{A}_{1}(\mathcal{P})\subset \mathcal{P}$, and the complete continuity of the operator $\mathcal{A}_{1}$ was verified in \cite{Dong,Dong2}. Also, due to Lemma \ref{lem 2.7}, we have the invariance property $\mathcal{A}_{2}(\mathcal{P})\subset \mathcal{P}$. Furthermore, the kernel $\frac{\lambda{t}}{2-\lambda{\eta^{2}}}\mathcal{H}(\eta,s)$ of $\mathcal{A}_{2}$ is continuous on $[0,1]\times[0,1]$, and using a standard argument, we can easily check that the operator $\mathcal{A}_{2}$ is also completely continuous.
Thus, we get the following lemma:

\begin{lemma} \label{lem 3.2}
If the hypothesis \rm{(H)} holds, then the operator $\mathcal{A}:\mathcal{P}\rightarrow \mathcal{P}$ is completely continuous.
\end{lemma}
The following lemma transforms the boundary value problem \eqref{eq1} and \eqref{eq2} into an equivalent fixed point problem.
\begin{lemma} \label{lem 3.3}
If the hypothesis \rm{(H)} holds, then the problem of nonnegative solutions of \eqref{eq1} and \eqref{eq2} is equivalent to the fixed point problem $x=\mathcal{A}x$, $x\in \mathcal{P}$.
\end{lemma}
\begin{proof}
It follows easily by using the same argument as for the proof of \emph{\cite[Lemma 3.3]{Zhong}}.
\end{proof}

For convenience, we introduce the following notations
\begin{align*}
f_{0} &= \lim_{x\rightarrow 0^{+}}\min_{t\in [0,1]} \frac{f(t,x)}{x},\ \ f^{\infty} = \lim_{x\rightarrow +\infty} \max_{t\in [0,1]}\frac{f(t,x)}{x},\\
f^{0} &= \lim_{x\rightarrow 0^{+}} \max_{t\in [0,1]}\frac{f(t,x)}{x},\ \ f_{\infty} = \lim_{x\rightarrow +\infty}\min_{t\in [0,1]} \frac{f(t,x)}{x},\\
\Lambda_{1}&= \Bigg(\theta^{4}\int_{\theta}^{1-\theta}\bigg(\mathcal{G}(s,s)
+\frac{\lambda}{2-\lambda\eta^{2}}\mathcal{H}(\eta,s)\bigg)ds\Bigg)^{-1},\ \ \Lambda_{2}=\Bigg(\bigg(1+\frac{\lambda}{2-\lambda\eta^{2}}\bigg)\int_{0}^{1}\mathcal{G}(s,s)ds\Bigg)^{-1}.
\end{align*}

Now, we will state and prove our main results.
\begin{theorem}\label{thm3.1}
Assume that the hypothesis \rm{(H)} holds. If $ f_{0} >\Lambda_{1} $ and $ f^{\infty}<\frac{\Lambda_{2}}{2}$, then the  problem \eqref{eq1} and \eqref{eq2} has at least one positive solution.
\end{theorem}
\begin{proof}

By Lemma \ref{lem 3.2}, we get that the operator $\mathcal{A}:\mathcal{P}\rightarrow \mathcal{P}$ is completely continuous.

Since $ f_{0} >\Lambda_{1} $, there exists  $ \rho_{1} > 0 $ such that $ f(t,x)\geq \Lambda_{1} x $, for $0 <x \leq \rho_{1}$ and $t\in[0,1].$
Thus
\[f(t,x(t))\geq \Lambda_{1} x(t) \  \text{for} \ t\in[0,1] \ \text{and} \ x\in \mathcal{P}\cap \partial\Omega_{\rho_{1}} .\]
By choosing $\varphi\equiv1$, it is obvious that $ \varphi\in \mathcal{P}\backslash \{0\}$. Now, we show that for the specified $\varphi$, the condition \rm{(C2)} in Lemma \ref{lem 2.8} is verified.
Assume that there exist a function $x_{0}\in \mathcal{P}\cap \partial\Omega_{\rho_{1}}$ and a positive number $\lambda_{0}$ such that
\[ x_{0}=\mathcal{A}x_{0}+\lambda_{0} \varphi.\]
Then, by Lemma \ref{lem 2.6} and \ref{lem 2.7}, for each $t\in[\theta,1-\theta]$, we have

\begin{equation*}
\begin{split}
x_{0}(t)&= \int_{0}^{1}\mathcal{K}(t,s)f(s,{x}_{0}(s))ds+\lambda_{0} \\
&= \int_{0} ^{1} \mathcal{G}(t,s)f(s,{x}_{0}(s)) ds+\frac{\lambda{t}}{2-\lambda{\eta^{2}}}\int_{0}^{1}\mathcal{H}(\eta,s)f(s,{x}_{0}(s))ds+\lambda_{0} \\
&\geq \theta^{2}\int_{\theta} ^{1-\theta} \mathcal{G}(s,s)\Lambda_{1}{x}_{0}(s)ds +\frac{\lambda{\theta^{2}}}{2-\lambda{\eta^{2}}}\int_{\theta}^{1-\theta}\mathcal{H}(\eta,s)\Lambda_{1}{x}_{0}(s)ds+\lambda_{0}\\
&\geq \theta^{2}\int_{\theta} ^{1-\theta} \mathcal{G}(s,s)\Lambda_{1}\theta^{2}\|{x}_{0}\|ds +\frac{\lambda{\theta^{2}}}{2-\lambda{\eta^{2}}}\int_{\theta}^{1-\theta}\mathcal{H}(\eta,s)\Lambda_{1}\theta^{2}\|{x}_{0}\|ds+\lambda_{0}\\
&=\|{x}_{0}\|\Lambda_{1}\Bigg(\theta^{4}\int_{\theta}^{1-\theta}\bigg(\mathcal{G}(s,s)
+\frac{\lambda}{2-\lambda\eta^{2}}\mathcal{H}(\eta,s)\bigg)ds\Bigg)+\lambda_{0}\\
&=\|{x}_{0}\|+\lambda_{0}.
\end{split}
\end{equation*}
Thus, $\|{x}_{0}\|\geq \|{x}_{0}\|+\lambda_{0}$. This is a contradiction. Hence the operator $\mathcal{A}$ satisfies the condition \rm{(C2)} in Lemma \ref{lem 2.8}.\\
We next show that the operator $\mathcal{A}$ satisfies the condition \rm{(C1)} in Lemma \ref{lem 2.8}.
The fact that $ f^{\infty}<\frac{\Lambda_{2}}{2}$ says us that there exists a constant $\gamma_{1}>0$ such that
\begin{equation}\label{eq17}
f(t,x)\leq \frac{\Lambda_{2}}{2} x \  \text{for} \ t\in[0,1] \ \text{and} \ x\geq \gamma_{1}.
\end{equation}
Define now
\begin{equation*}
\gamma_{2}=\max\{f(t,x):0\leq t\leq1,\ 0\leq x\leq \gamma_{1}\}.
\end{equation*}
So, by virtue of \eqref{eq17}, we get
\begin{equation}\label{eq18}
f(t,x)\leq \frac{\Lambda_{2}}{2} x+\gamma_{2} \  \text{for} \ t\in[0,1] \ \text{and} \ x\geq 0.
\end{equation}
Set $\rho_{2}=\max\{2\rho_{1},2\gamma_{2}\Lambda_{2}^{-1}\}$ and $x \in \mathcal{P}\cap \partial\Omega_{\rho_{2}}$. Then, by Lemma \ref{lem 2.6} and \eqref{eq18}, we obtain

\begin{equation*}
\begin{split}
\|\mathcal{A}x\|&= \max_{t \in [0,1]}\int_{0}^{1}\mathcal{K}(t,s)f(s,x(s))ds \\
&= \max_{t \in [0,1]}\Bigg\{\int_{0} ^{1} \mathcal{G}(t,s)f(s,x(s)) ds+\frac{\lambda{t}}{2-\lambda{\eta^{2}}}\int_{0}^{1}\mathcal{H}(\eta,s)f(s,x(s))ds \Bigg\}\\
&\leq \int_{0}^{1}\mathcal{G}(s,s)\bigg(\frac{\Lambda_{2}}{2} x(s)+\gamma_{2}\bigg)ds+\frac{\lambda}{2-\lambda{\eta^{2}}}\int_{0}^{1}\mathcal{G}(s,s)\bigg(\frac{\Lambda_{2}}{2} x(s)+\gamma_{2}\bigg)ds\\
&\leq \bigg(\frac{\Lambda_{2}}{2} \|x\|+\gamma_{2}\bigg) \bigg(1+\frac{\lambda}{2-\lambda{\eta^{2}}}\bigg)\int_{0}^{1}\mathcal{G}(s,s)ds \\
&= \frac{\|x\|}{2}+\gamma_{2}{\Lambda_{2}^{-1}}\\
&\leq  \frac{\|x\|}{2}+\frac{\|x\|}{2}\\
&=\|x\|.\\
\end{split}
\end{equation*}
Hence, the condition \rm{(C1)} in Lemma \ref{lem 2.8} is satisfied. By Lemma \ref{lem 2.8} and Lemma \ref{lem 3.3}, the operator $\mathcal{A}$ has at least one fixed point $ x\in\mathcal{P}\cap (\bar{\Omega}_{\rho_{2}}\backslash\Omega_{\rho_{1}})$, which is a positive solution of the boundary value problem \eqref{eq1} and \eqref{eq2}. The proof is complete.
\end{proof}

\begin{theorem}\label{thm3.2}
Assume that the hypothesis \rm{(H)} holds. If $ f^{0}<\Lambda_{2} $ and $ f_{\infty}>\Lambda_{1}$, then the  problem \eqref{eq1} and \eqref{eq2} has at least one positive solution.
\end{theorem}

\begin{proof}
We first note that, in virtue of Lemma \ref{lem 3.2}, the operator $\mathcal{A}$ is completely continuous.
Since $ f^{0}<\Lambda_{2} $ and $ f_{\infty}>\Lambda_{1}$, there exist two positive numbers $ \rho_1 >0$ and $\gamma_{1}>0$ such that

\begin{align}
f(t,x)&\leq \Lambda_{2} x,\ \text{for} \ t\in[0,1] \  \text{and} \ 0 <x \leq \rho_{1},\label{eq19}\\
f(t,x)&\geq \Lambda_{1} x ,\ \text{for} \ t\in[0,1] \  \text{and} \ x\geq \gamma_{1}\label{eq20}.
\end{align}
By \eqref{eq19} and Lemma \ref{lem 2.6}, for  $x \in \mathcal{P}\cap \partial\Omega_{\rho_{1}}$, we get
\begin{equation*}
\begin{split}
\|\mathcal{A}x\|&= \max_{t \in [0,1]}\int_{0}^{1}\mathcal{K}(t,s)f(s,x(s))ds\\
&\leq \Lambda_{2} \bigg(1+\frac{\lambda}{2-\lambda{\eta^{2}}}\bigg)\int_{0}^{1}\mathcal{G}(s,s)x(s)ds\\
&\leq \Lambda_{2} \|x\| \bigg(1+\frac{\lambda}{2-\lambda{\eta^{2}}}\bigg)\int_{0}^{1}\mathcal{G}(s,s)ds \\
&\leq \|x\|.
\end{split}
\end{equation*}
Thus the operator $\mathcal{A}$ satisfies the condition \rm{(C1)} in Lemma \ref{lem 2.8}.\\
Now, we show that the operator $\mathcal{A}$ also satisfies the condition \rm{(C2)} in Lemma \ref{lem 2.8}.
Let $\rho_{2}=\max\{2\rho_{1}, \gamma_{1}\theta^{-2}\}$, then by Lemma \ref{lem 2.7}, for $x \in \mathcal{P}\cap \partial\Omega_{\rho_{2}}$, we have
\[x(t)\geq \theta^{2}\rho_{2}\geq \gamma_{1}, \ \text{for}\ t\in[\theta,1-\theta].\]
Hence, by \eqref{eq20}, we have
\[f(t,x(t))\geq \Lambda_{1} x(t) ,\ \text{for} \ t\in[\theta,1-\theta]\ \text{and} \ x \in \mathcal{P}\cap \partial\Omega_{\rho_{2}}.\]
We now choose the function $\varphi\equiv1$, and clearly, $ \varphi\in \mathcal{P}\backslash \{0\}$. We then show that
\[x\neq \mathcal{A}x+\lambda \varphi, \ \text{for}\ x\in \mathcal{P}\cap \partial\Omega_{\rho_{2}} \ \text{and} \ \lambda>0.\]
If the above fact is not true, then there exist a function $x_{0}\in \mathcal{P}\cap \partial\Omega_{\rho_{2}}$ and a positive number $\lambda_{0}$ such that
\[ x_{0}=\mathcal{A}x_{0}+\lambda_{0} \varphi.\]
Then, by Lemma \ref{lem 2.6} and \ref{lem 2.7} , for each $t\in[\theta,1-\theta]$, we have

\begin{equation*}
\begin{split}
x_{0}(t)&= \int_{0}^{1}\mathcal{K}(t,s)f(s,{x}_{0}(s))ds+\lambda_{0} \\
&\geq \theta^{2}\int_{\theta} ^{1-\theta} \mathcal{G}(s,s)\Lambda_{1}{x}_{0}(s)ds +\frac{\lambda{\theta^{2}}}{2-\lambda{\eta^{2}}}\int_{\theta}^{1-\theta}\mathcal{H}(\eta,s)\Lambda_{1}{x}_{0}(s)ds+\lambda_{0}\\
&\geq\|{x}_{0}\|\Lambda_{1}\Bigg(\theta^{4}\int_{\theta}^{1-\theta}\bigg(\mathcal{G}(s,s)
+\frac{\lambda}{2-\lambda\eta^{2}}\mathcal{H}(\eta,s)\bigg)ds\Bigg)+\lambda_{0}\\
&=\|{x}_{0}\|+\lambda_{0}.
\end{split}
\end{equation*}
Thus, $\|{x}_{0}\|\geq \|{x}_{0}\|+\lambda_{0}$. This is a contradiction. Hence the operator $\mathcal{A}$ satisfies the condition \rm{(C2)} in Lemma \ref{lem 2.8}.\\
By Lemma \ref{lem 2.8} and Lemma \ref{lem 3.3}, the operator $\mathcal{A}$ has at least one fixed point $x\in\mathcal{P}\cap (\bar{\Omega}_{\rho_{2}}\setminus \Omega_{\rho_{1}})$, which is a positive solution of the boundary value problem \eqref{eq1} and \eqref{eq2}. The proof is complete.
\end{proof}
From Theorem \ref{thm3.1} and \ref{thm3.2}, we can obtain the following corollary.
\begin{corollary} \label{cor3.1}
Suppose that the hypothesis \rm{(H)} holds. If $f_{0}=\infty$ and $f^{\infty}=0$ or if $f^{0}=0$ and $f_{\infty}=\infty$, then the boundary value problem \eqref{eq1} and \eqref{eq2} has at least one positive solution.
\end{corollary}
\section{Examples}

\begin{example} Consider the following boundary value problem
\begin{equation}\label{eq4.1}
D^{\alpha}x(t) + t+e^{-x}= 0, \ t \in [0,1],
\end{equation}
\begin{equation}\label{eq4.2}
x(0) =0, \ x(1) =2 \int_{0}^{\frac{1}{3}}x(t)dt,
\end{equation}
where $\alpha\in(1,2]$, $\lambda=2$, $\eta=\frac{1}{3}$, and $ f(t,x)=t+e^{-x} \in C( [0, \infty), [0,\infty ))$,
so $ \lambda \eta^{2}=\frac{2}{9} < 2$. \\ We have

\[f_{0} = \lim_{x\rightarrow 0^{+}}\frac{e^{-x}}{x}= \infty, \
f^{\infty}=\lim_{x\rightarrow \infty } \frac{1+e^{-x}}{x}= 0.\]

Thus, by Corollary \ref{cor3.1}, the fractional boundary value problem \eqref{eq4.1}-\eqref{eq4.2} has at least one positive solution.
\end{example}

\begin{example} As a second example we consider the fractional boundary value problem
\begin{equation}\label{eq4.3}
D^{\frac{3}{2}}x(t) + t+\frac{4xe^{2x}/5}{e^{2x}+e^{x}-\frac{999}{500}}= 0, \ t \in [0,1],
\end{equation}
\begin{equation}\label{eq4.4}
x(0) =0, \ x(1) =\frac{8}{5} \int_{0}^{\frac{1}{2}}x(t)dt,
\end{equation}
where $\alpha=\frac{3}{2}$, $\lambda=\frac{8}{5}$, $\eta=\frac{1}{2}$, and $ f(t,x)=t+\frac{4xe^{2x}/5}{e^{2x}+e^{x}-\frac{999}{500}} \in C( [0, \infty), [0,\infty ))$,
so $ \lambda \eta^{2}=\frac{2}{5} < 2$. We have

\begin{align*}
f_{0} &= \lim_{x\rightarrow 0^{+}}\min_{t\in [0,1]} \frac{f(t,x)}{x} = \lim_{x\rightarrow 0^{+}} \frac{4e^{2x}/5}{e^{2x}+e^{x}-\frac{999}{500}} = 400,\\
f^{\infty} &= \lim_{x\rightarrow +\infty} \max_{t\in [0,1]}\frac{f(t,x)}{x} = \lim_{x\rightarrow +\infty } \bigg(\frac{1}{x}+\frac{4e^{2x}/5}{e^{2x}+e^{x}-\frac{999}{500}}\bigg) = \frac{4}{5}.
\end{align*}

By simple calculations, we find that
\[\Lambda_{2}^{-1}=\bigg(1+\frac{\lambda}{2-\lambda\eta^{2}}\bigg)\int_{0}^{1}\mathcal{G}(s,s)ds
=\frac{2}{\alpha(\alpha+1)}=\frac{8}{15}.\]
Hence, we get
\[ \Lambda_{2}=\frac{15}{8}>2f^{\infty}=\frac{8}{5}.\]
In addition, we have
\begin{equation*}
\begin{split}
\Lambda_{1}^{-1}&= \theta^{4}\int_{\theta}^{1-\theta}\bigg(\mathcal{G}(s,s)
+\mathcal{H}\bigg(\frac{1}{2},s\bigg)\bigg)ds \\
&=\theta^{4}\Bigg(\int_{\theta}^{1-\theta}(1-s)s^{\frac{1}{2}}ds+\int_{\theta}^{\frac{1}{2}}\mathcal{H}\bigg(\frac{1}{2},s\bigg)ds
+\int_{\frac{1}{2}}^{1-\theta}\mathcal{H}\bigg(\frac{1}{2},s\bigg)ds\Bigg)\\
&=\theta^{4}\Bigg(\frac{4}{5}\theta^{2}\sqrt{\theta}-\frac{7}{6}\theta\sqrt{\theta}
+\frac{1}{2}(1-\theta)\sqrt{1-\theta}-\frac{2}{5}(1-\theta)^{2}\sqrt{1-\theta}+\frac{1}{2}\sqrt{1-\theta}
-\frac{4}{15\sqrt{2}}\Bigg)\\
&=\frac{\theta^{4}}{30}\Bigg(\big(24\theta-35\big)\theta\sqrt{\theta}+3\big(6+3\theta-4\theta^{2}\big)\sqrt{1-\theta}
-4\sqrt{2}\Bigg).
\end{split}
\end{equation*}
By a Mathematica program, we easily check that $\Lambda_{1}<400=f_{0}$, for all $\theta\in[\frac{19}{50}, \frac{21}{50}]$. Therefore, all conditions of Theorem \ref{thm3.1} are fulfilled. Hence, problem \eqref{eq4.3}-\eqref{eq4.4} has at least one positive solution.
\end{example}


\end{document}